\documentclass[12pt, reqno]{amsart}
\usepackage{mathrsfs}
\usepackage[active]{srcltx}
\usepackage{mathrsfs,amsmath}
\usepackage{mathtools}
\usepackage{longtable}
\usepackage{todonotes}
\usepackage{amssymb} 
\usepackage[autostyle]{csquotes}
\allowdisplaybreaks

\UseRawInputEncoding

\usepackage{xspace}
\usepackage{xcolor}
\usepackage[normalem]{ulem}
\usepackage{etoolbox}

\newtoggle{ignoreflag}

\definecolor{bluecite}{HTML}{0875b7}
\usepackage[unicode=true,
bookmarksopen={true},
pdffitwindow=true,
colorlinks=true,
linkcolor=bluecite,
citecolor=bluecite,
urlcolor=bluecite,
hyperfootnotes=false,
pdfstartview={FitH},
pdfpagemode= UseNone]{hyperref}

\newcommand{\R}{{\mathbb R}}

\newtheorem{theorem}{Theorem}[section]
\newtheorem{lemma}{Lemma}[section]
\newtheorem{corollary}{Corollary}[section]

\newtheorem{remark}{Remark}[section]
\numberwithin{equation}{section}

\usepackage{bbm}  
\usepackage{dsfont}

\author{Zolt\'an M. Balogh,  Alexandru Krist\'aly and \'Agnes Mester}	

\address{Mathematisches Institute,
	Universit\"at Bern,
	Sidlerstrasse 5,
	3012 Bern, Switzerland}

\email{zoltan.balogh@unibe.ch}

\address{Department of Economics, Babe\c s-Bolyai University, Str. Teodor Mihali 58-60, 400591 Cluj-Napoca,
	Romania \& Institute of Applied Mathematics, \'Obuda
	University, B\'ecsi \'ut 96/B, 1034
	Budapest, Hungary}

\email{alexandru.kristaly@ubbcluj.ro; kristaly.alexandru@uni-obuda.hu}

\address{Mathematisches Institute,
	Universit\"at Bern,
	Sidlerstrasse 5,
	3012 Bern, Switzerland 
	\& Faculty of Mathematics and Computer Science, Babe\c s-Bolyai University, Mihail Kogalniceanu Str. 1,
	400084 Cluj-Napoca, Romania
	}

\email{agnes.mester@unibe.ch; agnes.mester@ubbcluj.ro
}

\thanks{Z. M. Balogh and \'A. Mester are
	supported by the Swiss National Science Foundation, Grant Nr. {200021\_228012}. 
	\thanks{A.\ Krist\'aly is  supported by the
	Excellence Researcher Program \'OE-KP-2-2022 of \'Obuda University, Hungary.} 
	}

\title[Sobolev inequalities on minimal submanifolds]{
	$L^p$-Sobolev inequalities on minimal submanifolds}

\begin{document}
	\begin{abstract}
		  The  paper is devoted to proving Allard-Michael-Simon-type $L^p$-Sobolev inequalities $(p>1)$   with explicit constants in the setting of Euclidean minimal submanifolds of arbitrary codimension. Our results require separate discussions for the cases $p\geq 2$ and $1<p<2$, respectively. In particular, for $p\geq 2$, we obtain an asymptotically sharp and codimension-free Sobolev constant. Our argument is based on optimal mass transport theory on Euclidean submanifolds and also provides an alternative, unified proof of the recent isoperimetric inequalities of Brendle (\textit{J.\ Amer.\ Math.\ Soc., 2021}) and Brendle and Eichmair (\textit{Notices Amer. Math. Soc., 2024}). 
	\end{abstract}
	\maketitle 
	
\vspace{-0.4cm}
	
		\section{Introduction and Results}

Sobolev inequalities are of crucial importance in the theory of partial differential equations, relating the Lebesgue norm of a function to its Dirichlet energy which measures its variation. The most famous inequality in this direction is certainly the  classical $L^p$-Sobolev inequality in the Euclidean space $\mathbb R^n$.  Namely, let $n\geq 2$ and $p\in (1,n),$ where $p'=\frac{p}{p-1}$ is the dual exponent of $p$ and $p^\star=\frac{pn}{n-p}$ is its critical Sobolev exponent. Then one has that
	\begin{equation}\label{Sobolev-000}
		\displaystyle  { \left(\int_{\mathbb R^n} |f|^{p^\star}{d}x\right)^{1/p^\star}\leq
			AT(n,p)
			\left(\int_{\mathbb R^n} |\nabla f|^p { d}x\right)^{1/p}},\ \ \forall  f\in C_0^\infty(\mathbb R^n).
	\end{equation}
	In this inequality, the best Sobolev constant, computed first by Aubin \cite{Aubin} and Talenti \cite{Talenti}, is given by 
	$$AT(n,p)=\pi^{-\frac{1}{2}}n^{-\frac{1}{p}} \left(\frac{p-1}{n-p}\right)^{1/p'}\left(\frac{\Gamma(\frac{n}{2}+1)\Gamma(n)}{\Gamma(\frac{n}{p})\Gamma(\frac{n}{p'}+1)}
	\right)^{{1}/{n}}.$$  
	When $p\to 1$, \eqref{Sobolev-000} reduces to  the classical sharp isoperimetric inequality in functional form, namely, 
		\begin{equation}\label{isop-000}
		\displaystyle  { \left(\int_{\mathbb R^n} |f|^\frac{n}{n-1}{ d}x\right)^\frac{n-1}{n}\leq
			\frac{1}{n\omega_n^{{1}/{n}}}
			\int_{\mathbb R^n} |\nabla f| {d}x},\ \ \forall  f\in C_0^\infty(\mathbb R^n),
	\end{equation}
 where $\omega_n$ denotes the volume of the unit ball in $\mathbb R^n.$
	
Motivated by various problems arising in differential geometry, analysis and PDEs on manifolds, the theory of Sobolev inequalities on \textit{Euclidean submanifolds} attracted considerable attention. In this context, it is expected that, beside the usual terms in \eqref{Sobolev-000}, the mean curvature of the submanifold should also play a role. 

To be more precise, let $m,n \geq 1$ and $\Sigma\subset \mathbb R^{n+m}$ be a complete $n$-dimensional submanifold, possibly with boundary $\partial \Sigma$, $H$ the mean curvature vector of $\Sigma$,  $\nabla^\Sigma$ the gradient associated to $\Sigma$,  ${\rm vol}_\Sigma$  the natural canonical measure on $\Sigma$, and $\sigma_\Sigma$  the surface measure on $\partial \Sigma$ induced by ${\rm vol}_\Sigma$, when $\partial \Sigma\neq \emptyset$. 	
In 1969,  Bombieri, De Giorgi and Miranda \cite{BDGM} established the first universal Sobolev inequality on minimal surfaces (i.e., surfaces with vanishing mean curvature), similar to \eqref{Sobolev-000}. Shortly after their work, in the early seventies, Allard \cite{Allard} and Michael and Simon \cite{Michael-Simon} provided the first general Sobolev/isoperimetric inequality on Euclidean submanifolds, incorporating the mean curvature. Namely, Theorem 2.1 in \cite{Michael-Simon} states that if $\Sigma\subset U$ for some open set $U\subset \mathbb R^{n+m}$ such that $\partial \Sigma = \emptyset$, then for any non-negative function $g\in C^1(U)$ which vanishes outside a compact subset of $U$, one has that 
	\begin{equation}\label{Mich-Sim-estimate}
	\displaystyle  { \left(\int_{\Sigma} g^\frac{n}{n-1}d{\rm vol}_\Sigma\right)^\frac{n-1}{n}\leq
		\frac{4^{n+1}}{\omega_n^{{1}/{n}}}
		\int_{\Sigma} \left(|\nabla^\Sigma g|+|H|g\right) d{\rm vol}_\Sigma}.
\end{equation}
Let us note that the constant $\frac{4^{n+1}}{\omega_n^{{1}/{n}}}$ in \eqref{Mich-Sim-estimate} is far from being optimal (see \eqref{isop-000} for comparison). Nevertheless, its main advantage lies in being codimension-independent.

The problem of proving a Sobolev inequality of the type \eqref{Mich-Sim-estimate} with  {\it best constant} remained open until more recently when, in a series of papers, Brendle \cite{ Brendle-Toulouse, Brendle}, then Brendle and Eichmair \cite{BrendleEichmair23, BrendleEichmair24} established 
the following Sobolev/isoperimetric inequality on Euclidean submanifolds, under the additional assumption that $\Sigma$ is compact, possibly with boundary: for every positive smooth function $g$ on $\Sigma$, one has that   	
	\begin{equation}\label{Brendle-isop}
		   \left(\displaystyle\int_{\Sigma} g^\frac{n}{n-1}d{\rm vol}_\Sigma\right)^\frac{n-1}{n}\leq
			C(n,m)
			\left\{\int_{\Sigma} \sqrt{|\nabla^\Sigma g|^2+ |H|^2g^2} d{\rm vol}_\Sigma+\int_{\partial \Sigma}g\,d\sigma_\Sigma\right\},
	\end{equation}
	where \begin{equation}\label{C-n-m-0}
		C(n,m)=\max \left\{ \frac{1}{n}\left(\frac{m\omega_m}{(n+m)\omega_{n+m}}\right)^{\frac{1}{n}}, \frac{1}{n \omega_n^{\frac{1}{n}}}\right\} = 
		 \displaystyle \begin{cases}
		 	 \frac{1}{n \omega_n^{1/n}} & m \in \{1,2\} \\
		 	 \frac{1}{n}\left(\frac{m\omega_m}{(n+m)\omega_{n+m}}\right)^{\frac{1}{n}} & m \geq 3 
		 \end{cases}	.
	\end{equation}
	
These are remarkable breakthrough results, especially in codimensions $m=1$ and $m=2$, when the constant $C(n,m)$ coincides with the best Euclidean constant (see \eqref{isop-000}), therefore making it optimal.  
Let us furthermore note that, although for $m\geq 3$ the constant $C(n,m)$ in \eqref{Brendle-isop} is codimension-dependent, according to the Nash isometric embedding theorem \cite{Nash}, the codimension $m$  can be controlled from above by $m_n=\frac{3}{2}n(n+3)$ for $\Sigma$ compact and $m_n=\frac{n}{2}(3n^2+14n+9)$ for $\Sigma$ non-compact; in both cases, the Brendle constant $C(n,m) $ still produces smaller values than the Michael-Simon constant	$MS(n)\coloneqq\frac{4^{n+1}}{\omega_n^{{1}/{n}}}$ from \eqref{Mich-Sim-estimate}. 
	
Considering \eqref{Brendle-isop}, two kinds of $L^p$-Sobolev inequalities can be deduced in the setting of Euclidean submanifolds when $p>1$; for simplicity, we consider the case when $\partial \Sigma=\emptyset$. First, similarly to Cabr\'e and Miraglio \cite{Cabre-Miraglio}, by  placing $g:=|f|^{p^\star\left(1-\frac{1}{n}\right)}$  with $p>1$ into the slightly weaker inequality implied by \eqref{Brendle-isop}, namely,
		\begin{equation*}
			\displaystyle  { \left(\int_{\Sigma} g^\frac{n}{n-1}d{\rm vol}_\Sigma\right)^\frac{n-1}{n}\leq
				C(n,m)
				\int_{\Sigma} \left(|\nabla^\Sigma g|+|H|g\right)d{\rm vol}_\Sigma},
		\end{equation*}
		 then applying H\"older's inequality, 		
		a simple computation yields that
\begin{equation}\label{Sobolev-change-of-functions}
	\displaystyle   \left(\int_{\Sigma} |f|^{p^\star}d{\rm vol}_\Sigma\right)^\frac{1}{p^\star}\leq
	p^\star\frac{n-1}{n}	C(n,m)
	\left(\int_{\Sigma} |\nabla^\Sigma f|^pd{\rm vol}_\Sigma\right)^\frac{1}{p} + C(n,m)\left(\int_{\Sigma}|H|^p|f|^p d{\rm vol}_\Sigma\right)^\frac{1}{p}.
\end{equation}
Second, if $\Sigma$ is minimal, i.e.,  $H\equiv 0$, then \eqref{Brendle-isop} and \eqref{Sobolev-000} combined with a  rearrangement argument and the co-area formula  imply the
 $L^p$-Sobolev inequality  
 \begin{equation}\label{Sobolev-1}
 	\displaystyle  { \left(\int_{\Sigma} |f|^{p^\star}d{\rm vol}_\Sigma\right)^{1/p^\star}\leq
 		n\omega_n^\frac{1}{n}C(n,m)AT(n,p)
 		\left(\int_{\Sigma} |\nabla^{\Sigma} f|^p d{\rm vol}_\Sigma\right)^{1/p}},\ \ \forall  f\in C_0^\infty(\Sigma).
 \end{equation}
When $m \in \{1,2\}$, we have that $C(n,m)=\frac{1}{n\omega_n^{{1}/{n}}}$, thus \eqref{Sobolev-1} is  sharp and we recover the optimal constant from the Euclidean Sobolev inequality \eqref{Sobolev-000}. This inequality was stated by Brendle \cite[Theorem 5.8]{Brendle-Toulouse} for minimal hypersurfaces.  However, for larger codimensions $m \geq 3$, \eqref{Sobolev-change-of-functions} and \eqref{Sobolev-1} are not sharp. In fact, the constants formally blow-up, since $C(n,m)\to \infty$ when $m\to \infty$, for every fixed $n\in \mathbb N$ (see  
\eqref{C-n-m-0}).

Motivated by these observations, a natural question arises: can we obtain \textit{codimension-free} $L^p$-Sobolev inequalities on submanifolds, or  inequalities which are better than \eqref{Sobolev-change-of-functions} and \eqref{Sobolev-1} for higher codimension? The purpose of this paper is to address this question in the case of \textit{minimal submanifolds}, i.e., submanifolds with vanishing mean curvature $H$. Let $C_0^\infty(\Sigma)$ be the space of  compactly supported smooth functions on $\Sigma$. Our first result is  the  following $L^p$-Sobolev inequality for $p\geq 2$: 
\begin{theorem}\label{main-theorem_p>2}
	Let $n\geq 3$ and $m\geq 1$ be integers, $2 \leq p < n$, and let  $\Sigma$ be a complete $n$-dimensional minimal submanifold  of $\mathbb R^{n+m}$ without boundary.  Then, for every $f\in C_0^\infty(\Sigma)$, one has that
	\begin{equation}\label{main-inequality-p>2}
		\displaystyle  { \left(\int_{\Sigma} |f|^{p^\star}d{\rm vol}_\Sigma\right)^{1/p^\star}\leq
			S(n,p)
			\left(\int_{\Sigma} |\nabla^\Sigma f|^p d{\rm vol}_\Sigma\right)^{1/p}},
	\end{equation} 
where
	$$S(n,p)=\frac{p^\star}{n}\left(1-\frac{1}{n}\right){p^{-\frac{1}{p}}}(2\pi)^{-\frac{1}{2}}\left(\frac{e}{n}\right)^{\frac{1}{p'}-\frac{1}{2}} \left(\frac{\Gamma(n)}{\Gamma(n/p)}
	\right)^{{1}/{n}}.$$
\end{theorem}

\begin{remark}\rm 
(i) First, we note that the constant $S(n,p)$ above is \textit{codimension-free}. Furthermore, as expected, $S(n,p)>AT(n,p)$. However,  by using the Stirling--L\'anczos approximation for the Gamma function, one has that 
\begin{equation}\label{asump-A-S}
	\lim_{n\to \infty}\frac{S(n,p)}{AT(n,p)}=1.
\end{equation}
 In particular, this limit implies that inequality
 \eqref{main-inequality-p>2} is 
	\textit{asymptotically sharp} for minimal submanifolds when $n\to \infty$. After the proof of Theorem \ref{main-theorem_p>2}, we point out in Remark \ref{remark-loosing-sharpness} the technical obstacle that prevents us from obtaining the optimal Sobolev constant $AT(n,p)$, yielding instead the asymptotically sharp constant $S(n,p)$.  
	
	(ii)  For simplicity, let $p=2$. Recall  $m_n=\frac{3}{2}n(n+3)$ for $\Sigma$ compact and $m_n=\frac{n}{2}(3n^2+14n+9)$ for $\Sigma$ non-compact from the Nash isometric embedding theorem. By the increasing property of $s\mapsto \frac{\Gamma(s+a)}{\Gamma(s)},$ $s>0$ for every $a>0$, it follows that  for every $n\geq 3$, one has the following chain of inequalities for the Michael--Simon constant from \eqref{Mich-Sim-estimate}, the Brendle constant from \eqref {C-n-m-0} (with Nash codimension bound for $m$), our constant  $S(n,2)$ and  the Aubin--Talenti constant $AT(n,2)$, respectively: 
	$$	MS(n)=\frac{4^{n+1}}{\omega_n^{{1}/{n}}}>C(n,m_n)>S(n,2)=\frac{n-1}{\sqrt{n(n-2)}}AT(n,2)>AT(n,2).$$
	In addition, their asymptotic behaviors are: 
	$$\lim_{n\to \infty}\frac{MS(n)}{C(n,m_n)}=+\infty;\ \ \lim_{n\to \infty}\frac{C(n,m_n)}{S(n,2)}=+\infty\ \ {\rm and}\ \ \lim_{n\to \infty}\frac{S(n,2)}{AT(n,2)}=1.$$  

\end{remark}

As a counterpart of Theorem \ref{main-theorem_p>2}, we  also derive an $L^p$-Sobolev inequality for $1 < p \leq 2$. In this case the constant is not codimension-free, but still slightly better than those in \eqref{Sobolev-change-of-functions} and \eqref{Sobolev-1} for certain ranges of the codimension $m$ and the value $p\in (1,2]$. Namely, we have: 
\begin{theorem}\label{main-theorem_p<2}
	Let $n \geq 2$ and $m\geq 1$ be integers, and $\Sigma$ be a complete $n$-dimensional minimal submanifold  of $\mathbb R^{n+m}$ without boundary. Let $1<p< 2 $ if $n=2$, and $1<p\leq 2 $ when $n \geq 3$, respectively. Then, for every $f\in C_0^\infty(\Sigma)$,
\begin{equation}\label{main-inequality-p<2}
	\displaystyle  
	\left(\int_{\Sigma} |f|^{p^\star}d{\rm vol}_\Sigma\right)^{1/p^\star} \leq \tilde S(n,m,p)
	\left(\int_\Sigma |\nabla^\Sigma f|^p d{\rm vol}_\Sigma\right)^{1/p} ,
\end{equation} 
where
$$\tilde S(n,m,p)=\frac{p^\star}{n}\left(1-\frac{1}{n}\right){p^{-\frac{1}{p}}{p'}^{-\frac{1}{p'}}}\left(\frac{\omega_m \Gamma(\frac{m}{p'}+1)}{\omega_{n+m} \Gamma(\frac{n+m}{p'}+1)}\frac{\Gamma(n)}{\Gamma(\frac{n}{p})}\right)^\frac{1}{n}.
$$ 
\end{theorem}

\begin{remark}\rm 
	Comparing $\tilde S(n,m,p)$ with the constants from the Sobolev inequalities \eqref{Sobolev-change-of-functions} and \eqref{Sobolev-1}, one can observe that for any $n \geq 2$ and $m \geq 1$,  
	\begin{align*}
		\lim_{p\to 1} p^\star\left(1-\frac{1}{n}\right)	C(n,m) 
		& = \lim_{p\to 1}n\omega_n^\frac{1}{n}C(n,m)AT(n,p) \\
		& = 		C(n,m) < 
		\lim_{p\to 1}\tilde S(n,m,p)=\frac{1}{n}\left(\frac{\omega_m }{\omega_{n+m} }\right)^\frac{1}{n} .
	\end{align*}
	This shows that for values of $p$ close to $1$, the constant $\tilde S(n,m,p)$ is worse than the constants appearing in \eqref{Sobolev-change-of-functions} and \eqref{Sobolev-1}. 
  	For larger values of $p\in (1,2]$, however, the constant $\tilde S(n,m,p)$ improves. For instance, when $n = 3$, $m = 4$ and $p=3/2$, one finds that 
	\begin{equation*}
	\tilde S(n,m,p)<
		 p^\star\left(1-\frac{1}{n}\right)	C(n,m) \quad \text{ and  } \quad
		 \tilde S(n,m,p) <	n\omega_n^\frac{1}{n}C(n,m)AT(n,p).
	\end{equation*}
	
	In addition,  for $p=2$, a simple calculation shows that $
	\tilde S(n,m,2)=S(n,2)$ for every $m\geq 1$. In fact, in the particular case when  $p=2$, the above two theorems coincide, which can be summarized by the following result:
\end{remark}
 
\begin{corollary}\label{main-corollary}
	Let $n\geq 3$ and $m\geq 1$ be integers and  $\Sigma$ be a complete $n$-dimensional minimal submanifold  of $\mathbb R^{n+m}$ without boundary.  Then, for every  $f\in C_0^\infty(\Sigma)$, 
	\begin{equation}\label{main-inequality}
		\displaystyle  { \left(\int_{\Sigma} |f|^{2^\star}d{\rm vol}_\Sigma\right)^{1/2^\star}\leq
			S(n,2)
			\left(\int_{\Sigma} |\nabla^{\Sigma} f|^2 d{\rm vol}_\Sigma\right)^{1/2}},
	\end{equation} 
	where $$S(n,2)=\pi^{-\frac{1}{2}}\frac{n-1}{n(n-2)} \left(\frac{\Gamma(n)}{\Gamma(n/2)}
	\right)^{{1}/{n}}.$$
\end{corollary}

The proof of Theorems \ref{main-theorem_p>2} and \ref{main-theorem_p<2} relies on Optimal Mass Transportation (shortly, OMT) on Euclidean submanifolds. The idea originates from the seminal paper of Cordero-Erausquin,  Nazaret and Villani \cite{CE-N-Villani}, where sharp $L^p$-Sobolev inequalities are established by applying OMT in the Euclidean setting. The key tool in this framework is Brenier's theorem \cite{Brenier} and the resulting Monge--Amp\`ere equation, linking the normalized initial function and the Talentian bubble. 

Recent results have shown that the OMT approach can be effectively used to produce sharp Sobolev inequalities in non-Euclidean spaces as well, see, e.g., Balogh and Krist\'aly \cite{BK_Annalen, BK},  
Brendle and Eichmair \cite{BrendleEichmair23, BrendleEichmair24}, Cavalletti and Mondino \cite{C-M-inv, C-M-2}, and Krist\'aly \cite{Kristaly}.
However, in the setting of submanifolds, the situation is more delicate, as the normalized initial function is supported on the submanifold $\Sigma$, therefore the source measure with this density is not absolutely continuous. To overcome this difficulty, we apply a generalized version of Brenier's theorem from the recent work of Balogh and Krist\'aly \cite{BK}  for not necessarily compactly supported measures; the compactly supported measures have been considered first by Wang \cite{Wang}. In this setting, we obtain an integral version of the Monge--Amp\`ere equation, where the target density function has to be defined in a higher dimensional space. To address this issue, the target density function will be chosen to be a Talentian-type bubble with a suitable exponent. 
 
The paper is organized as follows. In Section \ref{section-2} we first recall the OMT result of Balogh and Krist\'aly \cite{BK}, which is a crutial element of our arguments (see Theorem \ref{OMT-theorem-submanifold} below). Then, we shall focus on the proofs of the main results, i.e., Theorems \ref{main-theorem_p>2} and \ref{main-theorem_p<2}, respectively. Finally, based again on Theorem \ref{OMT-theorem-submanifold}, in Section \ref{section-3} we provide an alternative proof for the isoperimetric inequality \eqref{Brendle-isop}, see Theorem \ref{Brendle-isop-thm}. 

\section{Proof of Theorems \ref{main-theorem_p>2} and \ref{main-theorem_p<2}}\label{section-2}

In the sequel, we shall use the  integration formula 
	\begin{equation}\label{Gamma-integration}
		\int_0^\infty(\lambda +r^\alpha)^{-\gamma}r^\beta dr=\lambda^{-\gamma+\frac{\beta+1}{\alpha}}\frac{\Gamma\left(\gamma-\frac{\beta+1}{\alpha}\right)\Gamma\left(\frac{\beta+1}{\alpha}\right)}{\alpha\Gamma\left(\gamma\right)},
	\end{equation}
	for every $\lambda>0, \alpha>1,\beta>-1$ and $\gamma>\frac{\beta+1}{\alpha}$, whose proof is elementary, based on a suitable change of variable and basic properties of the Beta and Gamma functions, see e.g., Andrews, Askey and Roy \cite{Special_Functions}.  
	
	As mentioned earlier, the proofs of our main theorems  are based on OMT arguments. In order to have a self-contained presentation, we recall below the main OMT result from Balogh and Krist\'aly \cite[Theorem 2.1]{BK}, which can be viewed as an integral version of the Brenier--McCann theorem (see McCann \cite{McCann} and Villani \cite{Villani}). A similar result was also recently obtained by Wang \cite{Wang} for compactly supported measures.

	\begin{theorem}{\rm (Balogh and Krist\'aly \cite{BK})}
		\label{OMT-theorem-submanifold} 
		Let $n\geq2$ and $m\geq 1$ be integers,  $\Sigma$ be a complete $n$-dimensional submanifold of $\mathbb R^{n+m}$, possibly with boundary,  and $\Omega\subseteq \mathbb R^{n+m}$ be an open set. Let $\mu$ and $\nu$ be Borel probability measures on $\Sigma$ and $\Omega$ which are absolutely continuous with respect to $d{\rm vol}_\Sigma$ and $d\mathcal L^{n+m},$ respectively.  
		Then there exist a measurable subset $A$ of the normal bundle $T^\perp \Sigma$ and a function $u:\Sigma \to \mathbb R \cup \{+\infty\}$ which is semiconvex in its effective domain and it is twice differentiable on the set ${\rm Pr}_{T^\perp \Sigma}(A)\subset \Sigma$, that will give rise to a map $\Phi:A\to \Omega$ given by 
		$$\Phi(x,v)=\nabla^\Sigma u(x)+v,\ \ (x,v)\in A,$$
		such that 
		\begin{itemize}
			\item[(i)]  Pythagorean's rule holds, i.e., $|\Phi(x,v)|^2=|\nabla^\Sigma u(x)|^2+|v|^2$ for every $(x,v)\in A;$
			\item[(ii)] $\mu({\rm Pr}_{T^\perp \Sigma}(A))=\nu(\Phi(A))=1;$ 
			
			\item[(iii)]   for every $(x,v)\in A$ the $n\times n$ matrix  $D_\Sigma^2u(x)-\langle II(x),v\rangle$ is symmetric and non-negative definite, and the determinant-trace inequality holds, i.e., 
			\begin{equation}\label{determinant-trace-0}
				{\rm det} D\Phi(x,v)= {\rm det}[D_\Sigma^2u(x)-\langle II(x),v\rangle]\nonumber \leq \left(\frac{\Delta_{\rm ac}^\Sigma u(x)-\langle H(x),v\rangle}{n}\right)^n,
			\end{equation}
			where $II: T\Sigma\times T\Sigma \to T^{\perp}\Sigma$ stands for the second fundamental form of $\Sigma$, and 
			 $\Delta_{\rm ac}^\Sigma u$ is the absolute continuous part of the distributional Laplacian $\Delta_\mathcal D^\Sigma u;$
			\item[(iv)] $\Delta_{\rm ac}^\Sigma u \leq  \Delta_{\mathcal D}^\Sigma u$ in the sense of distributions, i.e., $\int_{\Sigma} f \Delta_{\rm ac}^\Sigma u \leq \int_{\Sigma} f\Delta_{\mathcal D}^\Sigma u $  for any compactly supported smooth function  $f:\Sigma \to \R_+$. 
			\item[(v)] 
			if $F$ and $G$ are the density functions of the measures $\mu$ and $\nu$ with respect to the volume measures of $\Sigma$ and $\mathbb R^{n+m}$, then for  $\mu$-a.e.\ $x\in \Sigma$ we have the following integral version of  the 
			Monge--Amp\`ere equation 
			\begin{equation*}
				F(x)=\int_{A\cap T_x^\perp \Sigma}G(\Phi(x,v)){\rm det} D\Phi(x,v)dv. 
			\end{equation*}
		\end{itemize}
	\end{theorem}

\noindent	We are now ready to present the proofs of our main results.

	\subsection{Proof of Theorem \ref{main-theorem_p>2}}
	Let $f\in C_0^\infty(\Sigma)$ be a function such that  $$\int_\Sigma |f|^{p^\star}d{\rm vol}_\Sigma=1.$$ For simplicity, it is enough to consider the case when $f\geq 0.$ Let $\mu$ and $\nu$ be probability measures on $\Sigma$ and $\mathbb R^{n+m}$, defined as  
\begin{equation}\label{prob-measures}
	d\mu(x)=f^{p^\star}(x) d{\rm vol}_\Sigma(x) \ \ {\rm  and} \ \ d\nu(y)=c_{n,m,p}^{-1} (1+ |y|^{p'})^{-n-\frac{m}{p'}}dy.
\end{equation}
  	In particular, by \eqref{Gamma-integration}, we have that
	$$c_{n,m,p}=\int_{\mathbb R^{n+m}}(1+|y|^{p'})^{-n-\frac{m}{p'}} dy=(n+m)\omega_{n+m}\frac{\Gamma\left(\frac{n}{p}\right)\Gamma\left(\frac{n+m}{p'}\right)}{p'\Gamma\left(n+\frac{m}{p'}\right)}.$$
	
	According to Theorem \ref{OMT-theorem-submanifold}, there exists
	a measurable subset $A$ of the normal bundle $T^\perp\Sigma$ and a function $u: \Sigma \to \R \cup \{+\infty\}$, such that the mapping $\Phi: A \to \R^{n+m}$, 
	$
	\Phi(x,v)=\nabla^\Sigma u(x)+v,\ \ (x,v)\in A,
	$
	satisfies the properties stated in Theorem \ref{OMT-theorem-submanifold}. 
	In particular, since $\Sigma$ is a minimal submanifold, i.e., $H \equiv 0$, the determinant-trace inequality  from Theorem \ref{OMT-theorem-submanifold}/(iii) implies that for every $(x,v)\in A$,
	\begin{equation}\label{determinant-trace-H=0}
		0 \leq {\rm det} D\Phi(x,v) \leq \left(\frac{\Delta_{\rm ac}^\Sigma u(x)}{n}\right)^n.
	\end{equation} 
	
	Let us denote by $A_x\coloneqq A\cap T_x^\perp\Sigma$, for every $x \in \Sigma$.
	By the Monge--Amp\`ere equation (see Theorem \ref{OMT-theorem-submanifold}/(v)), one has for ${\rm vol}_\Sigma$-a.e.\ $x\in \Sigma$ that 
	\begin{equation}\label{MA-p-talenti}
		f^{p^\star}(x)=c_{n,m,p}^{-1}\int_{A_x}(1+ |\Phi(x,v)|^{p'})^{-n-\frac{m}{p'}}{\rm det} D\Phi(x,v)dv
		.
	\end{equation} 
	Raising this relation to the power $1/n$, the determinant-trace inequality \eqref{determinant-trace-H=0} implies for ${\rm vol}_\Sigma$-a.e.\ $x\in \Sigma$ that
	\begin{align}\label{log-int-talenti}
		\nonumber	f^\frac{p^\star}{n}(x)& = c_{n,m,p}^{-\frac{1}{n}}\left(\int_{A_x}(1+ |\Phi(x,v)|^{p'})^{-n-\frac{m}{p'}}{\rm det} D\Phi(x,v)dv\right)^\frac{1}{n}\\
		& \leq \nonumber c_{n,m,p}^{-\frac{1}{n}}\left(\int_{A_x}(1+ |\Phi(x,v)|^{p'})^{-n-\frac{m}{p'}}\left(\frac{\Delta_{\rm ac}^\Sigma u(x)}{n}\right)^n dv \right)^\frac{1}{n}\\
		& = c_{n,m,p}^{-\frac{1}{n}} \frac{\Delta_{\rm ac}^\Sigma u(x)}{n} \left(\int_{A_x}(1+ |\Phi(x,v)|^{p'})^{-n-\frac{m}{p'}}dv\right)^\frac{1}{n}. 
	\end{align}

Let $t\in (0,1)$ be a parameter with value to be determined later.  Since  $p\geq 2$, it follows that $1<p'=\frac{p}{p-1}\leq 2$. The concavity of $|\cdot|^\frac{p'}{2}$ combined with Pythagorean's rule (see Theorem \ref{OMT-theorem-submanifold}/(i)) give that
\begin{equation}\label{phi-concave}
	|\Phi(x,v)|^{p'}=\left(|\nabla^\Sigma u(x)|^2+|v|^2\right)^\frac{p'}{2}\geq t^{1-\frac{p'}{2}}\left|\nabla^\Sigma u(x)\right|^{p'}+(1-t)^{1-\frac{p'}{2}}|v|^{p'}.
\end{equation}
Based on this estimate, the integral formula \eqref{Gamma-integration} and a change of variable imply that 
\begin{align*}
	I(x)&\coloneqq \int_{A_x}(1+ |\Phi(x,v)|^{p'})^{-n-\frac{m}{p'}}dv\\ 
	&\leq \int_{\mathbb R^m}\left(1+ t^{1-\frac{p'}{2}}\left|\nabla^\Sigma u(x)\right|^{p'}+(1-t)^{1-\frac{p'}{2}}|v|^{p'}\right)^{-n-\frac{m}{p'}}dv\\
	&= m \omega_m \int_{0}^\infty \left(1+ t^{1-\frac{p'}{2}}\left|\nabla^\Sigma u(x)\right|^{p'}+(1-t)^{1-\frac{p'}{2}}\rho^{p'}\right)^{-n-\frac{m}{p'}} \rho^{m-1} d\rho\\
	&= m \omega_m \int_{0}^\infty \left(1+ t^{1-\frac{p'}{2}}\left|\nabla^\Sigma u(x)\right|^{p'}+r^{p'}\right)^{-n-\frac{m}{p'}} (1-t)^{(\frac{1}{2}-\frac{1}{p'})m} r^{m-1} dr\\
	&=\left(1+ t^{1-\frac{p'}{2}}\left|\nabla^\Sigma u(x)\right|^{p'}\right)^{-n}(1-t)^{(\frac{1}{2}-\frac{1}{p'})m}m\omega_m\frac{\Gamma(n)\Gamma(\frac{m}{p'})}{p'\Gamma(n+\frac{m}{p'})}.
	\end{align*}
Accordingly, let us introduce the following notation:
\begin{align*}
	C_{n,m,p,t} &=\left(c_{n,m,p}^{-1}(1-t)^{(\frac{1}{2}-\frac{1}{p'})m}m\omega_m\frac{\Gamma(n)\Gamma(\frac{m}{p'})}{p'\Gamma(n+\frac{m}{p'})}\right)^\frac{1}{n} \\
	&= \left(\frac{\omega_m \Gamma(\frac{m}{p'}+1) }{\omega_{n+m} \Gamma(\frac{n+m}{p'}+1)}  \frac{\Gamma(n)}{\Gamma(\frac{n}{p})}\right)^{\frac{1}{n}} \left(1-t\right)^{\frac{m}{n}\left(\frac{1}{2}-\frac{1}{p'}\right)}.
\end{align*}
Therefore, multiplying the estimate  \eqref{log-int-talenti}  with  $f^{p^
\star(1-\frac{1}{n})}(x)\left(1+ t^{1-\frac{p'}{2}}\left|\nabla^\Sigma u(x)\right|^{p'}\right)$ yields that for ${\rm vol}_\Sigma$-a.e.\ $x\in \Sigma$, one has
\begin{equation}\label{amit-integralni-kell}
	f^{p^\star}(x)\left(1+ t^{1-\frac{p'}{2}}\left|\nabla^\Sigma u(x)\right|^{p'}\right) \leq \frac{C_{n,m,p,t}}{n}{\Delta_{\rm ac}^\Sigma u(x)}f^{p^
	\star(1-\frac{1}{n})}(x).
\end{equation}

In the next step, we shall integrate this inequality over $\Sigma$. To this end, we first focus on the integral appearing on the right-hand side. By Theorem \ref{OMT-theorem-submanifold}/(iv),
the divergence theorem and H\"older's inequality, one has that
\begin{eqnarray*}
	\int_\Sigma{\Delta_{\rm ac}^\Sigma u}f^{p^\star(1-\frac{1}{n})}d{\rm vol}_\Sigma&\leq& \int_\Sigma{\Delta_{\mathcal D}^\Sigma u}f^{p^\star\left(1-\frac{1}{n}\right)}d{\rm vol}_\Sigma \\
	&=&-p^\star\left(1-\frac{1}{n}\right)\int_\Sigma f^{p^\star(1-\frac{1}{n})-1}\langle\nabla^\Sigma u, \nabla^\Sigma f \rangle d{\rm vol}_\Sigma\\&\leq&p^\star\left(1-\frac{1}{n}\right)\left(\int_\Sigma |\nabla^\Sigma f|^p d{\rm vol}_\Sigma\right)^\frac{1}{p} \left(\int_\Sigma f^{p^\star}|\nabla^\Sigma u|^{p'}  d{\rm vol}_\Sigma\right)^\frac{1}{p'},
\end{eqnarray*}
where we used the relation $$p'\left(p^\star(1-\frac{1}{n})-1\right)=p^\star.$$
Note that the integral term containing $|\nabla^\Sigma u|^{p'}$ is finite. Indeed, by the Monge--Amp\`ere equation \eqref{MA-p-talenti} and a change of variable, it follows by \eqref{Gamma-integration} that
\begin{eqnarray}\label{J-definition}
\nonumber	J\coloneqq\int_\Sigma f^{p^\star}|\nabla^\Sigma u|^{p'}  d{\rm vol}_\Sigma&=& c_{n,m,p}^{-1} \int_\Sigma|\nabla^\Sigma u(x)|^{p'} \int_{A_x}(1+|\Phi(x,v)|^{p'})^{-n-\frac{m}{p'}}{\rm det} D\Phi(x,v)dvd{\rm vol}_\Sigma\\&\leq& \nonumber
	c_{n,m,p}^{-1}\int_\Sigma \int_{A_x}|\Phi(x,v)|^{p'}(1+|\Phi(x,v)|^{p'})^{-n-\frac{m}{p'}}{\rm det} D\Phi(x,v)dvd{\rm vol}_\Sigma\\&=&
		c_{n,m,p}^{-1}\int_{\mathbb R^{n+m}}|y|^{p'}(1+|y|^{p'})^{-n-\frac{m}{p'}}dy = \frac{(n+m)(p-1)}{n-p} < +\infty.
\end{eqnarray}
Taking into account the above estimates, integrating \eqref{amit-integralni-kell} over $\Sigma$ yields that
\begin{equation*}
1+t^{1-\frac{p'}{2}}	J \leq  \frac{p^\star}{n}\left(1-\frac{1}{n}\right) C_{n,m,p,t}\left(\int_\Sigma |\nabla^\Sigma f|^p d{\rm vol}_\Sigma\right)^\frac{1}{p} 	J^\frac{1}{p'}.
\end{equation*}
Equivalently, we have that 
\begin{equation*}
\nonumber	1 \leq  \frac{p^\star}{n}\left(1-\frac{1}{n}\right) C_{n,m,p,t}\left(\int_\Sigma |\nabla^\Sigma f|^p d{\rm vol}_\Sigma\right)^\frac{1}{p} 	
	\frac{ J^\frac{1}{p'}}{1+t^{1-\frac{p'}{2}}	J}.
\end{equation*}
The maximization of the function $J \mapsto \frac{ J^\frac{1}{p'}}{1 + t^{1-\frac{p'}{2}} J}, \ 0\leq J \leq \frac{(n+m)(p-1)}{n-p}$ implies that 
\begin{equation}\label{J-Young}
	\frac{ J^\frac{1}{p'}}{1+t^{1-\frac{p'}{2}}	J}\leq \frac{t^{\frac{1}{2}-\frac{1}{p'}}}{p^\frac{1}{p}{p'}^\frac{1}{p'}}.
\end{equation}
Therefore, it follows that
\begin{equation} \label{non-sharp-estimate}
	1\leq K_{m,n,p'}(t)^\frac{1}{n}\frac{p^\star}{n}\left(1-\frac{1}{n}\right)\frac{1}{p^\frac{1}{p}{p'}^\frac{1}{p'}}\left(\frac{\Gamma(n)}{\Gamma(\frac{n}{p})}
	\right)^{\frac{1}{n}}\left(\int_\Sigma |\nabla^\Sigma f|^p d{\rm vol}_\Sigma\right)^\frac{1}{p}	,
\end{equation}
where 
$$K_{m,n,p'}(t)=\frac{\omega_m \Gamma(\frac{m}{p'}+1)}{\omega_{m+n} \Gamma(\frac{m+n}{p'}+1)}  \left((1-t)^mt^n\right)^{\frac{1}{2}-\frac{1}{p'}},\ \ t\in (0,1).$$

Next, we aim to obtain an $m$-independent estimate for the term  $K_{m,n,p'}(t)$, taking advantage of the fact that $t\in (0,1)$ can be chosen arbitrarily.
In fact, we observe that the minimum of $t\mapsto K_{m,n,p'}(t)$ is achieved for $t=\frac{n}{n+m}$, thus the previous inequality implies that  
	\begin{equation}\label{K-version}
			1\leq K_{m,n,p'}^\frac{1}{n}\frac{p^\star}{n}\left(1-\frac{1}{n}\right)\frac{1}{p^\frac{1}{p}{p'}^\frac{1}{p'}}\left(\frac{\Gamma(n)}{\Gamma(\frac{n}{p})}
		\right)^{\frac{1}{n}}\left(\int_\Sigma |\nabla^\Sigma f|^p d{\rm vol}_\Sigma\right)^\frac{1}{p},
	\end{equation}
	where 
	$$	K_{m,n,p'}=\frac{\omega_m \Gamma(\frac{m}{p'}+1)}{\omega_{m+n} \Gamma(\frac{m+n}{p'}+1)}\left(\frac{m^mn^n}{(m+n)^{m+n}}\right)^{\frac{1}{2}-\frac{1}{p'}}.$$
	We notice that the sequence 
	$m\mapsto K_{m,n,p'} $ is increasing, see Balogh and Krist\'aly \cite[Proposition A.1.]{BK}, and by the asymptotic property of the Gamma function $\Gamma(r+\alpha) \sim r^\alpha \Gamma(r)$ as $r \to \infty$ (where $\alpha \in \R$), it follows that 
	\begin{equation}\label{K-limit}
		\lim_{m\to \infty}K_{m,n,p'}={p'}^\frac{n}{p'}(2\pi)^{-\frac{n}{2}}\left(\frac{e}{n}\right)^{n(\frac{1}{p'}-\frac{1}{2})}.
	\end{equation}
	Combining relations \eqref{K-version} and \eqref{K-limit}, it follows that
	\begin{equation*} 
	1 \leq  (2\pi)^{-\frac{1}{2}}p^\frac{1}{p'}\frac{n-1}{n(n-p)}\left(\frac{e}{n}\right)^{\frac{1}{p'}-\frac{1}{2}} \left(\frac{\Gamma(n)}{\Gamma(n/p)}
	\right)^{{1}/{n}}\left(\int_\Sigma |\nabla^\Sigma f|^p d{\rm vol}_\Sigma\right)^\frac{1}{p},
	\end{equation*}
	which is \eqref{main-inequality-p>2} in the normalized case $\int_\Sigma |f|^{p^\star}d{\rm vol}_\Sigma=1$. The general case follows by usual rescaling.
	\hfill$\square$
	
\begin{remark}\rm \label{remark-loosing-sharpness}
	(i) 		Relation \eqref{asump-A-S} points out that the constant $S(n,p)$ in \eqref{main-inequality-p>2} is asymptotically sharp. In fact, a careful inspection of the proof of  Theorem \ref{main-theorem_p>2} shows that a slightly non-sharp estimate is provided  in \eqref{non-sharp-estimate} by means of \eqref{J-Young}. Indeed, the question of sharpness reduces to how accurately the value of $J$, defined in \eqref{J-definition}, can be estimated. Note that in the standard Euclidean space $\mathbb R^n$ (i.e., when no codimension is present, thus formally $m=0$), a similar argument as above allows us to \textit{compute explicitly} the value of $J$, which is $\frac{n(p-1)}{n-p}$,  providing in this way the sharp Sobolev constant $AT(n,p)$; this is due to the fact that $J$ can be given as an integral of a Talentian-type function, by using the 'pure' Monge--Amp\`ere equation  with a change of variables. In the case of submanifolds with $m\geq 1$, however, the value of $J$ can only be estimated as $J\leq \frac{(n+m)(p-1)}{n-p}$, see \eqref{J-definition}, derived from the more involved form of the integral Monge--Amp\`ere equation. Therefore, we cannot substitute this expression for $J$; instead, we can only provide a generic estimate for the term involving $J$, see \eqref{J-Young}, which slightly deteriorates the sharpness of the Sobolev constant. 
	
	(ii)  The determinant-trace inequality from Theorem \ref{OMT-theorem-submanifold}/(iii) should imply in principle  an $L^p$-Sobolev inequality of the form 	\eqref{Sobolev-change-of-functions} for not necessarily minimal submanifolds $\Sigma$. However,  we are not able to derive such a result. The difficulty comes from the  estimate of 	 the integral
	$$\int_{A_x}(1+ |\Phi(x,v)|^{p'})^{-n-\frac{m}{p'}}\left(\Delta_{\rm ac}^\Sigma u(x)-\langle H(x),v\rangle\right)^n dv,$$	
 	as no \textit{a priori} information is available either on the sign of $\Delta_{\rm ac}^\Sigma u(x)$ or on the size of the set $A_x$, $x\in \Sigma$. In fact, we only know that $\Delta_{\rm ac}^\Sigma u(x)-\langle H(x),v\rangle\geq 0$ for every $x\in \Sigma$ and $v\in A_x$, which seems to be insufficient for further  estimates. 
	\end{remark}

	\subsection{Proof of Theorem \ref{main-theorem_p<2}} The proof is similar to that of Theorem \ref{main-theorem_p>2}, therefore we shall focus on the differences. As before, we consider the  probability measures \eqref{prob-measures}, having also the Monge--Amp\`ere equation \eqref{MA-p-talenti} and the pointwise estimate \eqref{log-int-talenti} for ${\rm vol}_\Sigma$-a.e.\ $x\in \Sigma$, that is
	\begin{equation*}
		f^\frac{p^\star}{n}(x) \leq \frac{c_{n,m,p}^{-\frac{1}{n}}}{n} \Delta_{\rm ac}^\Sigma u(x) \left(\int_{A_x}(1+ |\Phi(x,v)|^{p'})^{-n-\frac{m}{p'}}dv\right)^\frac{1}{n}. 
	\end{equation*} 
	Now, since $p\leq 2$ (thus, $p'\geq 2$), instead of \eqref{phi-concave}, we can write 
	\begin{equation}\label{phi-estimate}
		|\Phi(x,v)|^{p'}=\left(|\nabla^\Sigma u(x)|^2+|v|^2\right)^\frac{p'}{2}\geq \left|\nabla^\Sigma u(x)\right|^{p'}+|v|^{p'}.
	\end{equation}
	By using relations \eqref{phi-estimate} and  \eqref{Gamma-integration}, one has that 
	\begin{align*}
		\tilde I(x) \coloneqq &\int_{A_x}(1+ |\Phi(x,v)|^{p'})^{-n-\frac{m}{p'}}dv\\ \leq& \int_{\mathbb R^m}\left(1+ \left|\nabla^\Sigma u(x)\right|^{p'}+|v|^{p'}\right)^{-n-\frac{m}{p'}}dv\\=&\left(1+ \left|\nabla^\Sigma u(x)\right|^{p'}\right)^{-n}m\omega_m\frac{\Gamma(n)\Gamma(\frac{m}{p'})}{p'\Gamma(n+\frac{m}{p'})}.
	\end{align*}
	Due to this estimate, we consider the constant  
	$$\tilde C_{n,m,p}=\left(c_{n,m,p}^{-1}m\omega_m\frac{\Gamma(n)\Gamma(\frac{m}{p'})}{p'\Gamma(n+\frac{m}{p'})}\right)^\frac{1}{n} = \left(\frac{\omega_m \Gamma(\frac{m}{p'}+1) }{\omega_{n+m} \Gamma(\frac{n+m}{p'}+1)}  \frac{\Gamma(n)}{\Gamma(\frac{n}{p})}\right)^{\frac{1}{n}}.$$
	Therefore, multiplying   \eqref{log-int-talenti}  with  $f^{p^
		\star(1-\frac{1}{n})}(x)\left(1+ \left|\nabla^\Sigma u(x)\right|^{p'}\right)$, we obtain for ${\rm vol}_\Sigma$-a.e.\ $x\in \Sigma$ that 
	\begin{equation}\label{amit-integralni-kell-0}
		f^{p^\star}(x)\left(1+ \left|\nabla^\Sigma u(x)\right|^{p'}\right) \leq  \frac{\tilde C_{n,m,p}}{n}{\Delta_{\rm ac}^\Sigma u(x)}f^{p^
			\star(1-\frac{1}{n})}(x).
	\end{equation}
	Recall that 
	$$J = \int_\Sigma f^{p^\star}|\nabla^\Sigma u|^{p'}  d{\rm vol}_\Sigma \leq (n+m)\frac{p-1}{n-p} <+\infty .$$
	Therefore, similarly to the proof of Theorem \ref{main-theorem_p>2}, an integration of \eqref{amit-integralni-kell-0} over $\Sigma$ yields that
	\begin{equation*}
		1 +	J  \leq \frac{p^\star}{n}\left(1-\frac{1}{n}\right) \tilde C_{n,m,p}\left(\int_\Sigma |\nabla^\Sigma f|^p d{\rm vol}_\Sigma\right)^\frac{1}{p} 	J^\frac{1}{p'},
	\end{equation*}
	which can be written equivalently as  
	\begin{equation*}
		1\leq \frac{p^\star}{n}\left(1-\frac{1}{n}\right) \tilde C_{n,m,p}\left(\int_\Sigma |\nabla^\Sigma f|^p d{\rm vol}_\Sigma\right)^\frac{1}{p} 	\frac{ J^\frac{1}{p'}}{1+J}.
	\end{equation*}
	We observe that the maximum of the function $J \mapsto \frac{ J^\frac{1}{p'}}{1+J}, \ 0\leq J \leq \frac{(n+m)(p-1)}{n-p}$ is achieved for $J = p-1 = \frac{p}{p'}$, thus 
	$$\frac{ J^\frac{1}{p'}}{1+	J}\leq \frac{ (p-1)^\frac{1}{p'}}{p} = \frac{1}{p^\frac{1}{p}{p'}^\frac{1}{p'}}.$$ 
	Hence, it follows that
	$$1\leq \tilde S(n,m,p)\left(\int_\Sigma |\nabla^\Sigma f|^p d{\rm vol}_\Sigma\right)^\frac{1}{p}	,$$
	where 
	$$\tilde S(n,m,p)=\frac{p^\star}{n}\left(1-\frac{1}{n}\right)\frac{1}{p^\frac{1}{p}{p'}^\frac{1}{p'}}\left(\frac{\omega_m \Gamma(\frac{m}{p'}+1)}{\omega_{n+m} \Gamma(\frac{n+m}{p'}+1)}\frac{\Gamma(n)}{\Gamma(\frac{n}{p})}\right)^\frac{1}{n}.$$
	This concludes the proof in the  normalized  case. 
	\hfill$\square$\\
	
	Observe that the OMT method developed by Cordero-Erausquin,  Nazaret and Villani \cite{CE-N-Villani} has also been applied by Castillon \cite{Castillon} to establish $L^p$-Sobolev inequalities on submanifolds. However, Castillon's approach introduces certain weights in the Sobolev inequalities,  which correspond to the Jacobians of the orthogonal projection  from the tangent space of the submanifold to a fixed $n$-dimensional linear subspace of the ambient space $\mathbb R^{n+m}.$ 
	
	Before closing this section, let us note that the difference between Theorems \ref{main-theorem_p>2} and \ref{main-theorem_p<2} is mainly technical. On the one hand, when $p\geq 2$ (hence $p'\leq 2$), one can exploit the concavity of $|\cdot|^{p'/2}$ combined with properties of the Gamma function, which provides a codimension-free Sobolev constant in front of the gradient term in  \eqref{main-inequality-p>2}. On the other hand, when $1<p\leq 2$, this codimension-free character of the corresponding term is lost, nevertheless, we still have an improvement of the terms  known so far. In addition, for $n \geq 3$ and $p=2$, the formally different results of Theorems \ref{main-theorem_p>2} and \ref{main-theorem_p<2} actually coincide, as  shown in Corollary \ref{main-corollary}.

\section{Unified proof of the isoperimetric inequality  \eqref{Brendle-isop} via OMT}\label{section-3}

We now present an alternative proof of the Michael--Simon--Sobolev  inequality \eqref{Brendle-isop}, based on Theorem \ref{OMT-theorem-submanifold} due to Balogh and Krist\'aly \cite{BK}. Initially, \eqref{Brendle-isop} was  established by Brendle \cite{Brendle, Brendle-Toulouse} using the Alexandrov--Bakelman--Pucci method.  Subsequently, Brendle and Eichmair \cite{BrendleEichmair23} provided an alternative proof applying OMT techniques (see also Brendle and Eichmair \cite{BrendleEichmair24}). Note that both of these approaches assumed the compactness of the minimal submanifold. By applying Theorem \ref{OMT-theorem-submanifold}, we eliminate the need for this compactness assumption. For the reader's convenience, we restate the result here:   

\begin{theorem}\label{Brendle-isop-thm}
	Let $n \geq 2$, $m \geq 1$, and $\Sigma$ be a complete $n$-dimensional submanifold of $\mathbb{R}^{n+m}$, possibly with boundary $\partial \Sigma$. Then for every non-negative $f\in C_0^\infty(\Sigma)$, one has
	$$ 
	 \bigg(\int_{\Sigma}f^{\frac{n}{n-1}} d{\rm vol}_\Sigma\bigg)^{\frac{n-1}{n}} \leq C(n,m) \left\{\int_{\Sigma}\sqrt{|\nabla^{\Sigma}f|^2+f^2|H|^2} d{\rm vol}_\Sigma + \int_{\partial\Sigma}f d\sigma_\Sigma \right\} ,$$
	where 
	 \begin{equation*}
	 	C(n,m)=\max \left\{ \frac{1}{n}\left(\frac{m\omega_m}{(n+m)\omega_{n+m}}\right)^{\frac{1}{n}}, \frac{1}{n \omega_n^{\frac{1}{n}}}\right\}. 
	 \end{equation*} 
	Moreover, the constant $C(n,m)$ is sharp for $m = 1$ and $m = 2$.
\end{theorem}

Note that if $\Sigma$ is compact, Theorem \ref{Brendle-isop-thm} can be applied to any smooth positive function on $\Sigma$; in particular, \eqref{Brendle-isop} follows directly from Theorem \ref{Brendle-isop-thm}.


For the proof, we first establish the following preliminary result:
\begin{lemma}\label{property_of_alpha}
	Let $n \geq 2$, $m \geq 1$, and $\rho:[0, \infty) \rightarrow[0, \infty)$ be a measurable function such that $\rho(s) = 0$ for $s>1$, satisfying 
	\begin{equation*}
		\int_{\bar{B}^{n+m}} \rho\left(|y|^2\right) d y = 1.
	\end{equation*}
	Denoting by  
	\begin{equation*}
		\alpha_\rho \coloneqq \sup_{z\in \R^n} \int_{\{v\in \R^m: |z|^2+|v|^2 \leq 1\}} \rho(|z|^2+|v|^2)dv,
	\end{equation*}
	we have that
	\begin{equation} \label{alpha_lb}
		\alpha_\rho  \geq \max \left\{\frac{m \omega_{m}}{(n+m)\omega_{n+m}}, \frac{1}{\omega_{n}}\right\} 
		= 
		\displaystyle 
		\begin{cases}
			\frac{1}{\omega_n}, & \text{ if } ~ m = 1 \\
			\frac{1}{\omega_n} = \frac{m\omega_m}{(n+m)\omega_{n+m}},  & \text{ if } ~ m = 2 \\
			\frac{m\omega_m}{(n+m)\omega_{n+m}}, & \text{ if } ~ m \geq 3 
		\end{cases}	.
	\end{equation}		
\end{lemma}

\begin{proof}
	First, we shall focus on the inequality on the left-hand side of the preceding relation.		
	Note that  the condition $\int_{\bar{B}^{n+m}} \rho\left(|y|^2\right) d y = 1$ can be equivalently written as 
	$$ (n+m) \omega_{n+m} \int_{0}^{1} \rho\left(t^{2}\right) t^{n+m-1} d t = 1. $$ 
	Using this relation, we can write 
	\begin{align*} 
		\alpha_\rho &= m \omega_{m}  \sup _{r \in[0,1]} \int_{0}^{1} \rho\left(r^{2}+t^{2}\right)  t^{m-1} d t 
		\geq m \omega_{m} \int_{0}^{1} \rho\left(t^{2}\right)  t^{m-1} d t  \\
		&\geq m \omega_{m}  \int_{0}^{1} \rho\left(t^{2}\right) t^{n+m-1} d t  
		= \frac{m \omega_{m}}{(n+m)\omega_{n+m}}.
	\end{align*}
For future reference, observe that there is no function $\rho$  that attains equality in the second inequality above. Therefore, to achieve a value of $\alpha_{\rho}$ approaching $\frac{m \omega_{m}}{(n+m)\omega_{n+m}}$, $\rho$ must be concentrated near $t=1$.

	Next, in order to prove the inequality in  \eqref{alpha_lb}, we shall use the fact that for every $n \geq 2$ and $m\geq1$, the closed unit ball in $\R^{n+m}$ may be represented as 
	$$ \bar{B}^{n+m}=\left\{(z, v) \in \R^{n+m}: z \in \bar{B}^{n}, ~ v \in \bar{B}^m(z, \sqrt{1-|z|^{2}}) \right\}, $$
	where $\bar{B}^m(z, \sqrt{1-|z|^{2}})$ denotes the closed ball in $\R^m$ with center $z$ and radius $\sqrt{1-|z|^{2}}$. 
	Consequently, we have that 
	\begin{align}\label{alpha_m1}
		1 &= \int_{\bar{B}^{n+m}} \rho(|y|^{2}) dy = \int_{\bar{B}^{n}} \left(\int_{\bar{B}^m(z, \sqrt{1-|z|^2})} \rho (|z|^{2}+|v|^{2}) d v\right) d z \nonumber \\
		& \leq
		\omega_{n} \cdot \sup_{z \in \bar{B}^n} \int_{\bar{B}^m(z, \sqrt{1-|z|^2})} \rho(|z|^{2}+|v|^{2}) d v \nonumber  = \omega_{n} \alpha_\rho .
	\end{align}
	Note that in order to achieve the equality $\alpha_{\rho} = \frac{1}{\omega_{n}}$, $\rho$ needs to be chosen 
	with the property that the function
	$$ z\mapsto \int_{\bar{B}^m(z, \sqrt{1-|z|^2})} \rho(|z|^{2}+|v|^{2}) d v $$
	is constant. We will demonstrate that this choice is indeed possible in the case when $m=1$. 
	
	The second part of relation \eqref{alpha_lb} can be verified using standard properties of the Beta and Gamma functions. Indeed, for $m=1$, Gautschi's inequality gives
	${(n+1)\omega_{n+1}} > {2}\omega_{n}.$
	For $m=2$, the standard recursion formula for the Gamma function yields 
	${(n+2)\omega_{n+2}} ={2\pi} \omega_n.$
	Finally, by Cerone \cite[Corollary 1, p.\ 79]{Cerone}, it follows that   
	${(n+m)\omega_{n+m}}{} <  m \omega_{m}\omega_{n}$ for every $m \geq  3.$	
\end{proof}


\begin{proof}[Proof of Theorem \ref{Brendle-isop-thm}]
Let $f\in C_0^\infty(\Sigma)$ be a non-negative function on $\Sigma$ such that 
	\begin{equation}\label{normalization}
		\int_\Sigma f(x)^{\frac{n}{n-1}}d{\rm vol}_\Sigma(x) = 1,
	\end{equation} 
	and 
	 $\rho:[0, \infty) \rightarrow[0, \infty)$ be a measurable function such that $\rho(s)=0$ when $s>1$ and  
	\begin{equation}\label{rho_density}
		\int_{\bar{B}^{n+m}} \rho\left(|y|^2\right) d y = 1,
	\end{equation} 
	where, as before, $\bar{B}^{n+m}=\left\{y \in \mathbb{R}^{n+m}\right.$ : $|y| \leq 1\}$ is the closed unit ball in $\mathbb{R}^{n+m}$.
	We consider the probability measures $\mu$ and $\nu$ on $\Sigma$ and $\R^{n+m}$, respectively, defined as
	$$d\mu(x) = f^\frac{n}{n-1}(x)d \text{vol}_\Sigma(x) \quad\text{ and }\quad d\nu(y) = \rho(|y|^2) dy.$$	
	
	According to Theorem \ref{OMT-theorem-submanifold}, there exist
	a measurable subset $A \subset T^\perp\Sigma$ and a function $u: \Sigma \to \R \cup \{+\infty\}$ such that the mapping $\Phi: A \to \R^{n+m}$ given by
	$
	\Phi(x,v)=\nabla^\Sigma u(x)+v
	$
	satisfies the properties stated in Theorem \ref{OMT-theorem-submanifold}.
	Let $A_x \coloneqq A \cap T_x^\perp \Sigma$ for every $x \in \Sigma$.
	By the integral version of the  Monge--Amp\`{e}re equation (see Theorem \ref{OMT-theorem-submanifold}/(v)), we have for ${\rm vol}_\Sigma$-a.e.\ $x\in \Sigma$ that 
	$$f^\frac{n}{n-1}(x) = \int_{A_x} \rho(|\Phi(x,v)|^2) \mathrm{det}D\Phi(x,v)dv.$$
	Note that since $\nu(\Phi(A)) = 1$ (see Theorem \ref{OMT-theorem-submanifold}/(ii)) and $\rho(s) = 0, \forall s >1$, it follows that 
	$$|\Phi(x,v)|^2 = |\nabla^{\Sigma}u(x)|^2 + |v|^2 \leq 1, \quad \mu-\text{a.e. }x \in \Sigma, \ v \in A_x,$$
	thus 
	$$|v|^2 \leq 1 - |\nabla^{\Sigma}u(x)|^2, \quad  
	\mu-\text{a.e. }x \in \Sigma, \  v\in A_x . $$
	Therefore, applying the determinant-trace inequality from Theorem \ref{OMT-theorem-submanifold}/(iii) yields that
	\begin{align*}
		f^\frac{n}{n-1}(x) &\leq \int_{A_x} \rho(|\Phi(x,v)|^2) \left(\frac{\Delta_{\text{ac}}^\Sigma u(x)-\langle H(x),v\rangle}{n}\right)^n dv \\
		&\leq \int_{A_x} \rho(|\Phi(x,v)|^2) \left(\frac{\Delta_{\text{ac}}^\Sigma u(x)+ |H(x)||v|}{n}\right)^n dv \\
		&\leq \left(\frac{\Delta_{\text{ac}}^\Sigma u(x)+ |H(x)|\sqrt{1 - |\nabla^{\Sigma}u(x)|^2}}{n}\right)^n \int_{A_x} \rho(|\nabla^{\Sigma}u(x)|^2 + |v|^2)  dv ,
	\end{align*}
	for ${\rm vol}_\Sigma$-a.e.\ $x\in \Sigma$.
	Similarly to Brendle and Eichmair \cite{BrendleEichmair23}, let us denote by 
	\begin{equation*}
		\alpha_\rho \coloneqq \sup_{z\in \R^n} \int_{\{v\in \R^m: |z|^2+|v|^2 \leq 1\}} \rho(|z|^2+|v|^2)dv,
	\end{equation*}
	which is a positive number depending on the density $\rho$.
	Hence, it follows that 
	\begin{equation*}
		f^\frac{n}{n-1}(x) \leq \alpha_\rho \cdot \left(\frac{\Delta_{\text{ac}}^\Sigma u(x)+ |H(x)|\sqrt{1 - |\nabla^{\Sigma}u(x)|^2}}{n}\right)^n,
	\end{equation*}
	for ${\rm vol}_\Sigma$-a.e.\ $x\in \Sigma$.
	Raising this estimate to the power $1/n$, then multiplying by $f(x)$ and integrating over $\Sigma$, it follows that 
	\begin{equation}\label{div_tetel_elotti}
		n \alpha_\rho^{-\frac{1}{n}} \int_\Sigma f^\frac{n}{n-1}  d{\rm vol}_\Sigma
		\leq \int_\Sigma \Delta_{\text{ac}}^\Sigma u f d{\rm vol}_\Sigma + \int_\Sigma  |H|f \sqrt{1 - |\nabla^{\Sigma}u|^2} d{\rm vol}_\Sigma. 
	\end{equation}
Since $f\in C_0^\infty(\Sigma)$ is non-negative,	property (iv) from  Theorem \ref{OMT-theorem-submanifold} yields that
	\begin{equation}\label{from-Prop-V}
		\int_\Sigma f{\Delta_{\rm ac}^\Sigma u}   
		\leq \int_\Sigma f \, \Delta_{\mathcal D}^\Sigma u.
	\end{equation}
	If $\partial \Sigma \neq \emptyset$, in order to estimate the right-hand side, let us consider the sequence of cut-off functions $(\varphi_k)_{k \in \mathbb{N}}$ such that  $0 \le \varphi_k \le 1$, $\varphi_k \to 1$ pointwise on $\Sigma$, $|\nabla^\Sigma \varphi_k| \leq Ck$ for some fixed constant $C>1$ independent of $k$, and whose supports exhaust $\Sigma$. In particular, we define $(\varphi_k)_{k}$ as
	$$
	\varphi_k(x)
	=
	\begin{cases}
		~ 1, & \text{if } \operatorname{dist}(x,\partial\Sigma) \ge \frac{2}{k}, \\[6pt]
		~ \text{a smooth transition between $1$ and $0$}, 
		& \text{if } \frac{1}{k} \le \operatorname{dist}(x,\partial\Sigma) \le \frac{2}{k}, \\[6pt]
		~ 0, & \text{if } \operatorname{dist}(x,\partial\Sigma) \le \frac{1}{k}.
	\end{cases}
	$$
	Since $\varphi_k f \in C_0^\infty(\Sigma \setminus \partial \Sigma)$ and $u$ is semiconvex,  the divergence theorem implies that
	\begin{align}\label{divergence-thm}	
		\int_\Sigma f \, \Delta_{\mathcal D}^\Sigma u &= \lim_{k \to \infty} \int_\Sigma \varphi_k f \,  \Delta_{\mathcal D}^\Sigma u = \lim_{k \to \infty} \int_\Sigma u \, \Delta^\Sigma(\varphi_k f) \, d\text{vol}_\Sigma \nonumber \\
		&= \lim_{k \to \infty} \left( -\int_{\Sigma} \langle \nabla^\Sigma (\varphi_k f), \nabla^\Sigma u \rangle d{\rm vol}_\Sigma \right) \nonumber \\
		&= \lim_{k \to \infty} \left(-\int_{\Sigma} f \langle \nabla^\Sigma \varphi_k, \nabla^\Sigma u \rangle d{\rm vol}_\Sigma - \int_{\Sigma} \varphi_k \langle \nabla^\Sigma f, \nabla^\Sigma u \rangle d{\rm vol}_\Sigma \right).
	\end{align}
	Recall that $|\nabla^\Sigma u(x)| \leq 1$ for $\mu$--a.e. $x\in \Sigma$. Then, by the dominated convergence theorem, we have
	\begin{equation*}
		\lim_{k \to \infty} \int_{\Sigma} \varphi_k \langle \nabla^\Sigma f, \nabla^\Sigma u \rangle d{\rm vol}_\Sigma =  \int_\Sigma
		\langle \nabla^\Sigma f, \nabla^\Sigma u\rangle
		\, d\mathrm{vol}_\Sigma .
	\end{equation*}
	Furthermore,
	$$
	\left| 	\int_\Sigma
	f \langle \nabla^\Sigma \varphi_k, \nabla^\Sigma u\rangle  d\mathrm{vol}_\Sigma  \right|
	\leq 	\int_{\Omega_k} f |\nabla^\Sigma \varphi_k| \, d\mathrm{vol}_\Sigma,
	$$  
	where $
	\Omega_k = \left\{ 1/k \leq \operatorname{dist}(x,\partial\Sigma) \leq 2/k \right\}.
	$
	Letting $k\to\infty$ and using the coarea formula, we obtain that
	$$
	\limsup_{k\to\infty}
	\left| 	\int_\Sigma
	f \langle \nabla^\Sigma \varphi_k, \nabla^\Sigma u\rangle  d{\rm vol}_\Sigma 	\right|
	\leq 	\int_{\partial\Sigma} f \, d\sigma_\Sigma .
	$$
	Consequently, \eqref{divergence-thm} yields the estimate
	\begin{equation*}
		\int_\Sigma f \Delta_{\mathcal D}^\Sigma u 
		\leq \int_{\partial\Sigma} f d{\sigma}_\Sigma   - \int_{\Sigma} \langle \nabla^\Sigma f, \nabla^\Sigma u \rangle d{\rm vol}_\Sigma .
	\end{equation*}
	Combining this with \eqref{div_tetel_elotti} and \eqref{from-Prop-V}, it follows that
	\begin{equation*}
		n \alpha_\rho^{-\frac{1}{n}} \int_\Sigma f^\frac{n}{n-1}  d{\rm vol}_\Sigma
		\leq \int_{\partial\Sigma} f d{\sigma}_\Sigma   + \int_{\Sigma} |\nabla^\Sigma f| \cdot |\nabla^\Sigma u| d{\rm vol}_\Sigma  + \int_\Sigma  |H|f \sqrt{1 - |\nabla^{\Sigma}u|^2} d{\rm vol}_\Sigma. 
	\end{equation*}
	Therefore, by applying the elementary inequality 
	$$ab+cd \leq \sqrt{a^2+c^2} \sqrt{b^2+d^2}, \quad \forall a,b,c,d \in \R, $$
	to the choices  $a= |\nabla^{\Sigma} f|, b= |\nabla^{\Sigma} u|, c= |Hf|,  d= \sqrt{1-|\nabla^{\Sigma} u|^2}$, and taking note of the normalization \eqref{normalization}, 
	we obtain that
	\begin{equation}\label{isop_ineq_with_alpha}
		n \alpha_\rho^{-\frac{1}{n}} \left(\int_\Sigma f^\frac{n}{n-1} d{\rm vol}_\Sigma \right)^\frac{n-1}{n} \leq 
		\int_{\Sigma} \sqrt{|\nabla^\Sigma f|^2 + f^2 |H|^2 } d{\rm vol}_\Sigma + \int_{\partial\Sigma} f d{\sigma}_\Sigma .
	\end{equation}
	In the case when $\partial \Sigma = \emptyset$, by a similar method we obtain \eqref{isop_ineq_with_alpha} without the boundary term on the right-hand side.

	Finally, in order to recover the constant from Theorem \ref{Brendle-isop-thm}, one needs to make a suitable choice of the density function $\rho$. Based on Lemma \ref{property_of_alpha}, we distinguish the cases $m=1$ and $m \geq 2$. 
	Firstly, when $m=1$, 
	let $\rho: [0, \infty) \to [0, \infty)$ be given by  
	\begin{equation}\label{rho-fgv}
		\rho(s) =  
		\begin{cases} 
			\frac{c}{\sqrt{1-s}}&  s \in [0,1) \\
			0 & s \geq 1
		\end{cases}, 
	\end{equation} 
	where $c$ is a positive constant such that 
	\begin{equation*}
		\int_{\bar{B}^{n+1}} \rho(|y|^{2}) dy = \int_{\bar{B}^{n}} \left(\int_{-\sqrt{1-|z|^2}}^{\sqrt{1-|z|^2}} \rho (|z|^{2}+v^{2}) d v\right) d z = 1 .
	\end{equation*}
	For every $z \in B^n$, we have by \eqref{rho-fgv} that
	\begin{align*}
		\frac{1}{c}\int_{-\sqrt{1-|z|^{2}}}^{\sqrt{1-|z|^{2}}} \rho\left(|z|^{2}+v^{2}\right) d v
		= \pi.
	\end{align*}
	Consequently, $c = \frac{1}{\pi \omega_n}$.
	On the other hand, by the previous argument, we have that
	\begin{equation*}
		1 = \int_{\bar{B}^{n}} \left(\int_{-\sqrt{1-|z|^2}}^{\sqrt{1-|z|^2}} \rho (|z|^{2}+v^{2}) d v\right) d z = 
		\omega_{n} \cdot \sup_{z \in \bar{B}^n} \int_{- \sqrt{1-|z|^2}}^{\sqrt{1-|z|^2}} \rho(|z|^{2}+v^{2}) d v \nonumber  = \omega_{n} \alpha_\rho ,
	\end{equation*}
	thus $\alpha_\rho = \frac{1}{\omega_{n}}$. Applying \eqref{isop_ineq_with_alpha} and Lemma \ref{property_of_alpha} completes the proof for the case $m=1$.


	If $m \geq 2$, let us consider the density functions $\rho_j: [0, \infty) \to [0, \infty)$, $j \in \mathbb{N}^*$, defined by
	\begin{equation*}
		\rho_j(s) =  
		\begin{cases} 
			c_j s^j&  s \in [0,1] \\
			0 & s > 1
		\end{cases}, \quad \text{ where } \quad 
		c_j = \frac{2j+n+m}{(n+m)\omega_{n+m}}
	\end{equation*} 
	is chosen such that 
	$$\int_{\bar{B}^{n+m}} \rho_j\left(|y|^2\right) d y = 
	1.$$ 
	In this case, we have that
	\begin{align*}
		\alpha_{\rho_j}
		&= \sup_{r \in [0,1]} \int_{\{v\in \R^m:~ r^2+|v|^2 \leq 1\}} \rho_j(r^2+|v|^2)dv 
		=  m \omega_m  \sup_{r \in [0,1]}\int_0^{\sqrt{1-r^2}} \rho_j(r^2+t^2) t^{m-1} dt \\
		&= m \omega_m c_j \sup_{r\in [0,1]} \int_0^{\sqrt{1-r^2}} (r^2+t^2)^j t^{m-1} dt 
		\leq m \omega_m \frac{c_j }{2j+m},
	\end{align*}
	where we used the estimate
	\begin{equation*}
		\int_0^{\sqrt{1-r^2}} (r^2+t^2)^j t^{m-1} dt 
		\leq \int_0^{\sqrt{1-r^2}} (r^2+t^2)^j (r^2+t^2)^\frac{m-2}{2} t ~ dt 
		 \leq \frac{1}{2}   \int_{0}^{1} u^{j+\frac{m}{2}-1} du = \frac{1}{2j+m}.
	\end{equation*}
	Therefore, it follows that for each $j \in \mathbb{N}\setminus \{0\}$, 
	$$\alpha_{\rho_j} \leq \frac{ m \omega_m}{(n+m)\omega_{n+m}} \frac{2j+n+m}{2j+m}. $$ 
	On the other hand, considering the lower bound in  \eqref{alpha_lb}, we obtain that
	\begin{equation*}
		\lim_{j \to \infty} \alpha_{\rho_j} = \frac{ m \omega_m}{(n+m)\omega_{n+m}}.
	\end{equation*}
	This yields the desired constant in Theorem \ref{Brendle-isop-thm} in the case when $m \geq 2$.
\end{proof}

	\begin{remark}\rm 
		Considering inequality \eqref{isop_ineq_with_alpha} and the lower bound \eqref{alpha_lb}, the best constant attainable by this method is indeed the value appearing in Theorem \ref{Brendle-isop-thm}. Due to Lemma \ref{property_of_alpha}, this yields the optimal isoperimetric constant  when $m=1$ or $m=2$.   However, when $m \geq 3$, a sharp isoperimetric inequality cannot be achieved by this construction.   
	\end{remark}
	
	\vspace{0.5cm}
	
	\noindent {\bf Acknowledgment.} The authors thank Alessio Figalli and Kai-Hsiang Wang  for their comments on an earlier version of the manuscript. We also thank the anonymous referee for a useful remark that led to an improvement in one of the proofs.


%


\end{document}